\documentclass[11pt, twoside, a4paper]{article}
\usepackage[english]{babel}
\usepackage{amsmath,amssymb,amsthm,epsfig}
\newcommand{\keywords}[1]{\par\addvspace\baselineskip\noindent\textbf{Keywords:}\enspace\ignorespaces#1}

\newcommand{\AMSclassification}[1]{\par\addvspace\baselineskip\noindent\textbf{Mathematical subject classification:}\enspace\ignorespaces#1}

\numberwithin{equation}{section}

\newtheorem{theorem}{Theorem}
\newtheorem{proposition}[theorem]{Proposition}
\newtheorem{lemma}[theorem]{Lemma}

\newtheorem{definition}{Definition}
\newtheorem{notation}[definition]{Notation}
\newtheorem{remark}[definition]{Remark}

\begin{document}

\title{Weak KAM methods and ergodic optimal problems \\ for countable Markov shifts}
\author{Rodrigo Bissacot\\
\footnotesize{Departamento Matem\'atica, UFMG, 30161-970 Belo Horizonte -- MG, Brasil}\\
\footnotesize{\texttt{rodrigo.bissacot@gmail.com}}
\and
Eduardo Garibaldi\\
\footnotesize{Departamento de Matem\'atica, UNICAMP, 13083-859 Campinas -- SP, Brasil}\\
\footnotesize{\texttt{garibaldi@ime.unicamp.br}}}
\date{\today}
\maketitle

\begin{abstract}
Let $ \sigma : \boldsymbol{\Sigma} \to \boldsymbol{\Sigma} $ be the left shift acting on $ \boldsymbol{\Sigma} $, a one-sided Markov subshift on a 
countable alphabet. Our intention is to guarantee the existence of $\sigma$-invariant Borel probabilities that maximize the integral of a
given locally H\"older continuous potential $ A : \boldsymbol{\Sigma} \to \mathbb R $. Under certain conditions, we are able to show
not only that $A$-maximizing probabilities do exist, but also that they are characterized by the fact their support lies actually in a 
particular Markov subshift on a finite alphabet. To that end, we make use of objects dual to maximizing measures, the so-called sub-actions 
(concept analogous to subsolutions of the Hamilton-Jacobi equation), and specially the calibrated sub-actions (notion similar to
weak KAM solutions).


\AMSclassification{37A05, 37A60, 37B10.}
\end{abstract}

\begin{section}{Introduction}

The development of the study of maximizing probabilities has given place to a new and exciting field in ergodic theory.
Growing in the intersection of topological dynamical systems and optimization theory, this fresh theorical branch is known nowadays as
\emph{ergodic optimization}. Many results were already obtained for dynamics defined by a continuous map $ T: X \to X $ 
of a compact metric space $ X $ assuming $ T $ has some hyperbolicity (see, for instance, \cite{CG, CLT, GL, Jenkinson, LT}).
Although ergodic optimal problems in the context of noncompact dynamical systems have been much less discussed, interesting 
works can be found in the literature (see, for example, \cite{JMU1, JMU2}).

The principal purpose of this article is to take into account ergodic optimal problems for a class of noncompact symbolic dynamics: topological Markov shifts
with a countable number of states. Let then $ \boldsymbol{\Sigma} $ denote a one-sided Markov subshift on a countable alphabet, and  
$ \sigma : \boldsymbol{\Sigma} \to \boldsymbol{\Sigma} $ the left shift map. If $ A : \boldsymbol{\Sigma} \to \mathbb R $ is continuous and 
bounded above, one would like to determine and describe the $\sigma$-invariant Borel probability measures $ \mu $ that maximize the average value 
$ \int A \; d\mu $. In general, such a maximizing probability does not even exist, since $ \boldsymbol{\Sigma} $ may be noncompact. 
We show that this is not the case when the potential $ A $ is sufficiently regular and verifies a coercive condition. 
In reality, our main theorem (see theorem~\ref{existence maximizing}) states that, for one of these 
specific potentials, its maximizing probabilities have in common the fact of being supported in a certain compact 
$\sigma$-invariant subset that is actually contained in a Markov subshift on a finite alphabet.

A second objective of this paper is to point out that weak KAM methods (or viscosity solutions technics) can be adapted and employed also in noncompact 
ergodic optimization. Tools of the theory of viscosity solutions have been successfully used in Lagrangian mechanics (see, for instance, \cite{CI, Fathi}). 
Ergodic optimization on compact spaces has witnessed the usefulness of these methods, specially when ergodic optimal problems are interpreted as questions of
variational dynamics (see, for example, \cite{CLT, GL, LT}). We adopt the same spirit and strategy here. 

\end{section}

\begin{section}{Basic concepts and main result}

Our dynamical setting will be special topologically mixing Markov subshifts on a countable alphabet: the \emph{primitive} ones. 
Let us introduce them precisely. 

For the sake of definiteness, the countably infinite alphabet will always be the set of nonnegative integers $ \mathbb Z_+ $. 
Let thus $ \mathbf M : \mathbb Z_+ \times \mathbb Z_+ \to \{0, 1 \} $ be a transition matrix. 
Consider the following sets of symbols given in an inductive way by
$$ \mathcal B_0 = \{ i \in \mathbb Z_+ : \mathbf M(i,j) = 1 \text{ for some } j \in \mathbb Z_+\} \;\; \text{ and } $$
$$ \mathcal B_n = \{ i \in \mathbb Z_+ : \mathbf M(i,j) = 1 \text{ for some } j \in \mathcal B_{n - 1} \}, \;\; \text{ for } \; n > 0. $$
We say that the transition matrix $ \mathbf M $ is \emph{primitive} if there exist 
a (possibly countable) subset $ \mathbb F \subseteq \mathbb Z_+ $ and an integer $ K_0 \geq 0 $ such that, 
for any pair of symbols $ i, j \in \bigcap_{n \ge 0} \mathcal B_n $, one can find $ \ell_1, \ell_2, \ldots, \ell_{K_0} \in \mathbb F $ satisfying
$$ \mathbf M(i, \ell_1) \mathbf M(\ell_1, \ell_2) \cdots \mathbf M(\ell_{K_0}, j) = 1. $$
In particular, we say that $ \mathbf M $ is \emph{finitely primitive} when $ \mathbb F $ is finite.

Consider then the associated Markov subshift
$$ \boldsymbol{\Sigma} = \left \{ \mathbf x = (x_0, x_1, \ldots) \in {\mathbb Z_+}^{\mathbb Z_+} : \mathbf M(x_j, x_{j + 1}) = 1 \right \}. $$
Fixed $ \lambda \in (0, 1) $, we equip $ \boldsymbol{\Sigma} $ with the complete 
metric $ d(\mathbf x, \mathbf y) = \lambda^k $, where
$ \mathbf x = (x_0, x_1, \ldots), \mathbf y = (y_0, y_1, \ldots) \in \boldsymbol{\Sigma} $ and $ k = \min \{j: x_j \ne y_j \} $.
It is easy to see that $ \boldsymbol{\Sigma} $ is compact if, and only if, $ \bigcap_{n \ge 0} \mathcal B_n $ is finite\footnote{Since the compact situation is 
well studied, the interesting case occurs naturally when $ \bigcap_{n \ge 0} \mathcal B_n $ is countable. The reader is thus invited to assume this hypothesis 
without hesitation.}. Let $ \sigma : \boldsymbol{\Sigma} \to \boldsymbol{\Sigma} $ be the shift map, namely, 
$ \sigma(x_0, x_1, x_2, \ldots) = (x_1, x_2, \ldots) $. We will also say that the dynamics $ (\boldsymbol{\Sigma}, \sigma) $ is (finitely) primitive. 
Since $ \mathbf M $ is primitive, clearly $ (\boldsymbol{\Sigma}, \sigma) $ is a topologically mixing dynamical system.

Denote by $ \mathcal M_\sigma $ the $\sigma$-invariant Borel probability measures. 
Let $ C^0(\boldsymbol{\Sigma}) $ indicate the space of continuous real-valued functions on $ \boldsymbol{\Sigma} $, 
equipped with the topology of uniform convergence on compact subsets. 
We remind then central concepts in the ergodic optimization theory.

\begin{definition}
If the potential $ A \in C^0(\boldsymbol{\Sigma}) $ is bounded above, we define the ergodic maximizing value by
$$ \beta_A = \sup_{\mu \in \mathcal M_\sigma} \int A \; d\mu. $$
Any $\sigma$-invariant probability achieving this supremum is called maximizing (or, if precision is required, $A$-maximizing). 
\end{definition}

We are particularly interested in ergodic optimal results for \emph{locally H\"older continuous} potentials.

\begin{definition}
A potential $ A : \boldsymbol{\Sigma} \to \mathbb R $ is called locally H\"older continuous
when there exists a constant $ H_A > 0 $ such that, for all integer $ k \ge 1 $, we have
$$ \text{Var}_k (A) := \sup_{\mathbf x, \mathbf y \in \boldsymbol{\Sigma}, \; d(\mathbf x, \mathbf y) \le \lambda^k} 
\left [ A(\mathbf x) - A(\mathbf y) \right ] \; \leq \; H_A \lambda^{k}. $$
\end{definition}

Such a regularity condition only means that the $k$-th variation $ \text{Var}_k (A) $ decays exponentially fast to
zero when $ k \rightarrow \infty $. We could focus on more general regularity assumptions, like summability of variations. 
Recall that $ A : \boldsymbol{\Sigma} \to \mathbb R $ has summable variations if
$$ \text{Var} (A) := \sum_{k = 1}^\infty \text{Var}_k (A) < \infty. $$
Yet one of our main goals here is to provide examples of the applicability of the weak KAM technics.
We believe local H\"older continuity is sufficient for this end. 

Notice that nothing is required from $ \text{Var}_0 (A) := \sup_{\mathbf x, \mathbf y \in \boldsymbol{\Sigma}} [A(\mathbf x) - A(\mathbf y)] $, 
which means that a locally H\"older continuous potential, despide its uniform continuity, may be unbounded. So a common assumption\footnote{Note 
this assumption is trivially verified when $ \mathbb F $ is finite. Indeed, choosing a point $ \mathbf x^i \in [i] $ for each $ i \in \mathbb F $, 
if $ \mathbf x \in [j] $, $ j \in \mathbb F $, then $ A(\mathbf x) > A(\mathbf x^j) - \text{Var}_1(A) $ obviously implies 
$ \inf  A|_{\bigcup_{i \in \mathbb F} [i]} > \min_{i \in \mathbb F} A(\mathbf x^i) - \text{Var}_1(A) > - \infty $.}\label{obviedade} 
in this article will be
$$ \inf  A|_{\bigcup_{i \in \mathbb F} [i]} > - \infty, $$
where $ [i] $ just indicates the cylinder set $ \{\mathbf x = (x_0, x_1, \ldots) \in \boldsymbol{\Sigma} : x_0 = i\} $. 

Under this hypothesis, we will obtain in the next section a dual formula
$$ \beta_A = \inf_{f \in C^0(\boldsymbol{\Sigma})} \sup_{\mathbf x \in \boldsymbol{\Sigma}} \left(A + f - f\circ\sigma \right)(\mathbf x). $$

This expression raises the natural question about the existence of
functions achieving the above infimum, which motivates the following definition.

\begin{definition}
Suppose $ A : \boldsymbol{\Sigma} \to \mathbb R $ is continuous and bounded above. 
A sub-action (for the potential $ A $) is a function $ u \in C^0(\boldsymbol{\Sigma}) $ verifying 
$$ (A + u - u\circ\sigma)(\mathbf x) \le \beta_A, \;\; \forall \; \mathbf x \in \boldsymbol{\Sigma}. $$ 
\end{definition}

We will see in section~\ref{maximizing} that it is possible to construct locally H\"older continuous sub-actions for
potentials with the same regularity (see proposition~\ref{minimal sub-action}). This result is completely new as far as
we know. 

In the context of a noncompact dynamical system, given an arbitrary bounded above continuous potential, the existence of maximizing probabilities is
a nontrivial question. However, we will be able to use the existence of sub-actions as well as their properties in order to guarantee there exist maximizing
probabilities when we are taking into account coercive potentials.

\begin{definition}
A continuous potential $ A: \boldsymbol{\Sigma} \to \mathbb R $ is said coercive when
$$ \lim_{i \to + \infty} \sup A|_{[i]} = - \infty. $$
\end{definition}

In Aubry-Mather theory for Lagrangian systems, superlinearity is the usual coercive hypothesis (see, for instance, \cite{CI, Fathi}). 
The coercive condition is not strange to the countable Markov shift framework. On the contrary, it is an essential theorical piece 
(in general implicitly) in several studies of the thermodynamic formalism generalized to a finitely primitive Markov subshift on a 
countable alphabet. Coerciveness obviously follows from the imposition $ \sum_{i} \exp{(\sup A|_{[i]})} < \infty $. This summability condition is 
equivalent to the finiteness of the topological pressure when the potential $ A $ is, for example, locally H\"older continuous. 
This summability condition also allows to define the Ruelle operator 
$ \mathcal L_A f(\mathbf x) := \sum_{\sigma(\mathbf y) = \mathbf x} e^{A(\mathbf y)} f(\mathbf y) $ 
for a bounded continuous function $ f : \boldsymbol{\Sigma} \to \mathbb R $. For more details, we refer the reader to the book of 
R. D. Mauldin and M. Urba\'nski (see \cite{MU}). Furthermore, when $ (\boldsymbol{\Sigma}, \sigma) $ is finitely primitive and $ A $ is locally H\"older 
continuous, it is not difficult to show the hypothesis $ \| \mathcal L_A 1 \|_{\infty} < \infty $ 
(omnipresent in the work of O. Sarig \cite{Sarig}) implies coerciveness too. 

Given a nonnegative integer $ I $, denote by
$$ \Sigma_I = \left\{\mathbf x = (x_0, x_1, \ldots) \in \{0, \ldots, I\}^{\mathbb Z_+} : \mathbf M(x_j, x_{j + 1}) = 1 \right\} $$
the Markov subshift on the finite alphabet $ \{ \iota_1, \ldots, \iota_{r_I}\} := \{0, \ldots, I\} \cap \left( \bigcap_{n \ge 0} \mathcal B_n \right) $ 
associated to the transition matrix $ \mathbf M|_{\{0, \ldots, I\} \times \{0, \ldots, I\}} $. 
Obviously $ \Sigma_I $ is a compact $\sigma$-invariant subset of $ \boldsymbol{\Sigma}$. 
So we simply denote $ \sigma |_{\Sigma_I} $ by $ \sigma $.

When $ \mathbf M $ is finitely primitive, let
$$ I_{\mathbb F} := \max \{i : i \in \mathbb F\}. $$ 
Our main result concerning the existence of maximizing probabilities can be stated as follows. 

\begin{theorem}\label{existence maximizing}
Suppose $ (\boldsymbol{\Sigma}, \sigma) $ is a finitely primitive Markov subshift on a countable alphabet. Let 
$ A: \boldsymbol{\Sigma} \to \mathbb R $ be a bounded above, coercive and locally H\"older continuous potential. 
Then there exists an integer $ \hat I > I_{\mathbb F} $ such that 
$$ \beta_A = \max_{\substack{\mu \in \mathcal M_\sigma \\ \text{supp} \mu \subseteq \Sigma_{\hat I}}} \int A \; d \mu. $$
In particular, maximizing measures do exist. Furthermore, there exists a compact $\sigma$-invariant set $ \Omega \subseteq \Sigma_{\hat I} $
such that $ \mu \in \mathcal M_\sigma $ is an $A$-maximizing probability if, and only if, $ \mu $ is supported in $ \Omega $. 
\end{theorem}

Its proof is discussed in section~\ref{maximizing} and exploits the analogy with Aubry-Mather theory in symbolic dynamics. 
The existence of a bounded continuous sub-action for the potential $ A $ will tell us where to seek maximizing probabilities. 
Nevertheless, a key step to the demonstration is to analyse first the problem for the compact situation $ (\Sigma_I, \sigma) $, 
using the uniform oscillatory behavior of some special sub-actions, that are called calibrated,
and should be understood as corresponding to Fathi's weak KAM solutions or viscosity solutions of the Hamilton-Jacobi equation.

Theorem~\ref{existence maximizing} clarifies previous results. For instance, in the special case where $ \sum_{i} \exp{(\sup A|_{[i]})} < \infty $,
the identity  $ \beta_A = \max_{\mu \in \mathcal M_\sigma, \, \text{supp} \mu \subseteq \Sigma_{\hat I}} \int A \; d \mu $
is implicitly present in the work of I. D. Morris. Indeed, in the proof of lemma 3.5 of \cite{Morris}, one obtains that, if $ \{\mu_t\}_{t > 1} $ 
is the family of equilibrium states of $ tA $, then there is $ \hat I \in \mathbb Z^+ $ such that $ \mu_t([i]) \to 0 $ as $ t \to \infty $
for all $ i > \hat I $. Since this family of probabilities is uniformly tight and any accumulation measure is maximizing, this shows that
an $A$-maximizing probability exists and is supported in $ \Sigma_{\hat I} $. Concerning the description of all $A$-maximizing probabilities,
in \cite{JMU1} the authors obtained in a more general context a not so precise characterization for their supports 
(see remark~\ref{melhoria}). 

\end{section}

\begin{section}{Characterizations of the ergodic maximizing value}

We will present other expressions which one could choose in
order to introduce the constant $ \beta_A $ for our particular situation. 
In this section, we will consider a larger class of potentials: the uniformly continuous ones.
Remind that $ A : \boldsymbol{\Sigma} \to \mathbb R $ is uniformly continuous if $ \lim_{k \to \infty} \text{Var}_k (A) = 0 $.
Notice we are still dealing with functions which may be unbounded.
 
Given $ A \in C^0(\boldsymbol{\Sigma}) $, as usual let $ S_k A = \sum_{j = 0}^{k - 1} A \circ \sigma^j $ and $ S_0 A = 0 $. 
Hence, the following result identifies the ergodic maximizing value with a maximum ergodic time average.

\begin{proposition}
Let $ (\boldsymbol{\Sigma}, \sigma) $ be a primitive Markov subshift on a countable alphabet. 
Assume the uniformly continuous potential $ A : \boldsymbol{\Sigma} \to \mathbb R $ is bounded above and satisfies 
$ \inf A|_{\bigcup_{i \in \mathbb F} [i]} > - \infty $. Then we verify
$$ \beta_A = \lim_{k \to \infty} \sup_{\mathbf x \in \boldsymbol{\Sigma}} \frac{1}{k} S_k A(\mathbf x) =
\inf_{k \ge 1} \sup_{\mathbf x \in \boldsymbol{\Sigma}} \frac{1}{k} S_k A(\mathbf x). $$
\end{proposition}

\begin{proof}
Note that $ \{ \sup_{\mathbf x \in \boldsymbol{\Sigma}} S_k A(\mathbf x) \}_{k \ge 1} $ is a subadditive sequence of real numbers.
Therefore, the limit $ \lim_{k \to \infty} \sup_{\mathbf x \in \boldsymbol{\Sigma}} \frac{1}{k} S_k A(\mathbf x) $
exists and is in fact equal to $ \inf_{k \ge 1} \sup_{\mathbf x \in \boldsymbol{\Sigma}} \frac{1}{k} S_k A(\mathbf x) $.

Given a positive integer $ k $, take a point $ \mathbf x^k \in \boldsymbol{\Sigma} $ satisfying
$$ \sup_{\mathbf x \in \boldsymbol{\Sigma}} \frac{1}{k} S_k A(\mathbf x) - \frac{1}{2^k} < \frac{1}{k} S_k A(\mathbf x^k). $$
Since $ (\boldsymbol{\Sigma}, \sigma) $ is a primitive Markov subshift, for all sufficiently large $ k $,
we can find a periodic point $ \mathbf y^k = (y_0^k, y_1^k, \ldots) \in \boldsymbol{\Sigma} $ of period $ k $, with $ y_j^k \in \mathbb F $ 
for each $ j \in \{k - K_0, \ldots, k - 1\} $, such that $ d(\mathbf x^k, \mathbf y^k) \le \lambda^{k - K_0} $. 
From the immediate inequality
$$ \frac{1}{k} S_k A(\mathbf y^k) \le \beta_A, $$
we obtain
\begin{align*}
\frac{1}{k} S_k A(\mathbf x^k) & \le \frac{1}{k} S_k A(\mathbf x^k) - \frac{1}{k} S_k A(\mathbf y^k) + \beta_A \\
& \le \frac{1}{k} \left[ \text{Var}_{k - K_0}(A) + \ldots + \text{Var}_1 (A) + K_0 \left(\sup A - \inf A|_{\bigcup_{i \in \mathbb F} [i]} \right) \right]
+ \beta_A.
\end{align*}
For $ k $ large enough, we thus have
$$ \sup_{\mathbf x \in \boldsymbol{\Sigma}} \frac{1}{k} S_k A(\mathbf x) - \frac{1}{2^k} <
\frac{1}{k} \left[\sum_{j=1}^k \text{Var}_j (A) + K_0 \left(\sup A - \inf A|_{\bigcup_{i \in \mathbb F} [i]} \right) \right] + \beta_A. $$ 
So $ \lim_{k \to \infty} \sup_{\mathbf x \in \boldsymbol{\Sigma}} \frac{1}{k} S_k A(\mathbf x) \le \beta_A $.

In order to show the equality does hold, take a probability $ \mu \in \mathcal M_\sigma $ such that $ A \in L^1(\mu) $.
For any $ k > 0 $, we clearly have 
$$ \int A \, d\mu = \int \frac{1}{k} S_k A \, d\mu \le \sup_{\mathbf x \in \boldsymbol{\Sigma}} \frac{1}{k} S_k A(\mathbf x). $$
Taking the infimum over $ k $ and then the supremum over $ \mu $, we finish the proof.
\end{proof}

We remark that, for a noncompact dynamical system, in general we have
$$ \beta_A \le \limsup_{k \to \infty} \sup_{\mathbf x \in \boldsymbol{\Sigma}} \frac{1}{k} S_k A(\mathbf x). $$
We refer the reader to \cite{JMU2} for a discussion on such a topic.

We present now a dual characterization of $ \beta_A $.

\begin{proposition}
Let $ (\boldsymbol{\Sigma}, \sigma) $ be a primitive Markov subshift on a countable alphabet. 
Suppose the uniformly continuous potential $ A : \boldsymbol{\Sigma} \to \mathbb R $ is bounded above and verifies 
$ \inf A|_{\bigcup_{i \in \mathbb F} [i]} > - \infty $. Then
$$ \beta_A = \inf_{f \in C^0(\boldsymbol{\Sigma})} \sup_{\mathbf x \in \boldsymbol{\Sigma}} \left(A + f - f\circ\sigma \right)(\mathbf x). $$
\end{proposition}

\begin{proof}
Denote by $ C^0_A(\boldsymbol{\Sigma}) $ the set of continuous functions $ f : \boldsymbol{\Sigma} \to \mathbb R $ satisfying
$ \sup (A + f - f\circ\sigma) < \infty $. Note that all bounded continuous real-valued functions belong to $ C^0_A(\boldsymbol{\Sigma}) $. 
Moreover, we clearly have
$$ \inf_{f \in C^0(\boldsymbol{\Sigma})} \sup_{\mathbf x \in \boldsymbol{\Sigma}} \left(A + f - f\circ\sigma \right)(\mathbf x) =
\inf_{f \in C^0_A(\boldsymbol{\Sigma})} \sup_{\mathbf x \in \boldsymbol{\Sigma}} \left(A + f - f\circ\sigma \right)(\mathbf x) < \infty. $$

By conciseness, write 
$ \varkappa = \inf_{f \in C^0_A(\boldsymbol{\Sigma})} \sup_{\mathbf x \in \boldsymbol{\Sigma}} \left(A + f - f\circ\sigma \right)(\mathbf x) $.
Fix $ \epsilon > 0 $. Choose a function $ f \in C^0_A(\boldsymbol{\Sigma}) $ such that $ A + f - f\circ\sigma < \varkappa + \epsilon $.
For any $ \mu \in \mathcal M_\sigma $, we verify
$$ \int A \; d\mu = \int (A + f - f\circ\sigma) \; d\mu \le \varkappa + \epsilon. $$
Hence, $ \beta_A \le \varkappa + \epsilon $. Since $ \epsilon > 0 $ is arbitrary, we get $ \beta_A \le \varkappa $.

Consider then $ f_k = - \frac{1}{k} \sum_{j = 1}^k S_j A \in C^0(\boldsymbol{\Sigma}) $. The identity
$$ A = \frac{1}{k} S_k (A\circ\sigma) + f_k\circ\sigma - f_k $$ 
implies $ \sup (A + f_k - f_k\circ\sigma) = \sup \frac{1}{k} S_k (A\circ\sigma) \le \sup A < \infty $,
that is, $ f_k \in C^0_A(\boldsymbol{\Sigma}) $. Therefore, we obtain
$$ \varkappa \le \inf_{k \ge 1} \sup_{\mathbf x \in \boldsymbol{\Sigma}} \frac{1}{k} S_k (A\circ\sigma)(\mathbf x). $$
The result follows thus from the previous proposition.
\end{proof}

In ergodic optimization on compact spaces, a similar dual expression of the corresponding ergodic maximizing value is well known
(see, for example, \cite{CG}).

\end{section}

\begin{section}{Sub-actions and maximizing probabilities}\label{maximizing}

\begin{subsection}*{A minimal sub-action}

We will show the existence of minimal sub-actions for locally H\"older continuous potentials.
Similar results have been obtained in the compact situation (see, for example, \cite{CLT, GL}).

\begin{proposition}\label{minimal sub-action}
Assume $ (\boldsymbol{\Sigma}, \sigma) $ is a primitive Markov subshift on a countable alphabet. 
Let $ A : \boldsymbol{\Sigma} \to \mathbb R $ be a bounded above and locally H\"older continuous potential such that $ \inf A|_{\bigcup_{i \in \mathbb F} [i]} > - \infty $. 
Then there exists an unique minimal, nonnegative, bounded and locally H\"older continuous function $ u_A: \boldsymbol{\Sigma} \to \mathbb R_+ $ verifying 
$$ A + u_A - u_A\circ\sigma \le \beta_A. $$ 
The minimality is in the sense that, for any nonnegative sub-action $ u \in C^0(\boldsymbol{\Sigma}, \mathbb R_+) $ 
(not necessarily locally H\"older continuous), we have $ u_A \leq u $.
\end{proposition}

\begin{proof}
Given $ \mathbf x \in \boldsymbol{\Sigma} $, define
$$ u_A(\mathbf x) := \sup \left\{ S_{k}(A - \beta_A)(\mathbf y) : k \geq 0, \; \mathbf y \in \boldsymbol{\Sigma}, \; \sigma^{k}(\mathbf y) = \mathbf x \right\}. $$

As $ S_0 (A - \beta_A) = 0 $ by convention, obviously $ u_A \ge 0 $. 

Take an integer $ k > K_0 $ and a point $ \mathbf y \in \boldsymbol{\Sigma} $ verifying $ \sigma^{k}(\mathbf y) = \mathbf x $. 
We can thus find a periodic point $ \mathbf y^k = (y_0^k, y_1^k, \ldots ) \in \boldsymbol{\Sigma} $ of period $ k $,
with $ y_j^k \in \mathbb F $ when $ j \in \{k - K_0, \ldots, k - 1\} $, such that $ d(\mathbf y, \mathbf y^k) \le \lambda^{k - K_0} $. 
First notice that
$$ S_k A(\mathbf y) - S_k A(\mathbf y^k) \leq \text{Var}_{k - K_0} (A) + \ldots + \text{Var}_1(A) + 
K_0 \left( \sup A - \inf A|_{\bigcup_{i \in \mathbb F} [i]} \right). $$
Since clearly $ S_k A (\mathbf y^k) \leq k \beta_A $, we then obtain
$$ S_k (A - \beta_A)(\mathbf y) \le \text{Var}(A) + K_0 \left( \sup A - \inf A|_{\bigcup_{i \in \mathbb F} [i]} \right), \quad \forall \; k > K_0, $$
which assures that
\begin{equation}\label{boundness uA}
0 \le u_A(\mathbf x) \le \max \left\{ \text{Var}(A) + K_0 \left( \sup A - \inf A|_{\bigcup_{i \in \mathbb F} [i]} \right), \; K_0 (\sup A - \beta_A) \right\}. 
\end{equation}

So $ u_A : \boldsymbol{\Sigma} \to \mathbb R_+ $ is a well defined bounded function. 
Moreover, from the identity $ A\circ\sigma^k + S_k (A - \beta_A) = S_{k + 1} (A - \beta_A) + \beta_A $ and the definition of $ u_A $, we get 
$ A + u_A \le u_A\circ\sigma + \beta_A $.

Concerning its regularity, $ u_A $ is a locally H\"older continuous function. Indeed, let
$ \mathbf x = (x_0, x_1, \ldots), \bar{\mathbf x} = (\bar x_0, \bar x_1, \ldots) \in \boldsymbol{\Sigma} $ be arbitrary points with
$ d(\mathbf x, \bar{\mathbf x}) \le \lambda^k $ for some $ k \geq 1 $. Given $ \epsilon > 0 $, take an integer $ \bar k \ge 0 $ and a point 
$ \bar{\mathbf y} = (\bar y_0, \bar y_1, \ldots) \in \boldsymbol{\Sigma} $, with $ \sigma^{\bar k}(\bar{\mathbf y}) = \bar{\mathbf x} $, such that
$$ u_A(\bar{\mathbf x}) - \epsilon < S_{\bar k} (A - \beta_A)(\bar{\mathbf y}). $$
Consider the point $ \mathbf y = (\bar y_0, \bar y_1, \ldots, \bar y_{\bar k - 1}, x_0, x_1, \ldots) \in \boldsymbol{\Sigma} $ satisfying 
$ \sigma^{\bar k}(\mathbf y) = \mathbf x $. So we have
\begin{eqnarray*} 
u_A(\bar{\mathbf x}) - u_A(\mathbf x) - \epsilon & < & 
S_{\bar k} A(\bar{\mathbf y}) - S_{\bar k} A(\mathbf y) \\
& \le & 
\text{Var}_{k + \bar k} (A) + \text{Var}_{k + \bar k - 1} (A) + \ldots + \text{Var}_{k}(A) \\
& \le &
H_A \left( \lambda^{k + \bar k} + \lambda^{k + \bar k - 1} + \ldots + \lambda^k \right) \\
& \le &
\frac{H_A}{1 - \lambda} \lambda^k.
\end{eqnarray*}
Since $ \epsilon $ can be considered arbitrarily small, this shows that 
$$ \text{Var}_k (u_A) \le \frac{H_A}{1 - \lambda} \lambda^k, $$
which means $ u_A $ is locally H\"older continuous (with constant $ H_{u_A} = \frac{H_A}{1-\lambda} $). 

Suppose now that $ u \in C^0(\boldsymbol{\Sigma}, \mathbb R_+) $ is a nonnegative sub-action for the potential $ A $.
Given $ \mathbf x \in \boldsymbol{\Sigma} $, if the point $ \mathbf y \in \boldsymbol{\Sigma}  $ satisfies
$ \sigma^k(\mathbf y) =\mathbf x $ for some $ k \ge 0 $, it is easy to see that $ u(\mathbf x) + k \beta_A \ge S_k A(\mathbf y) + u(\mathbf y) \ge S_k A(\mathbf y) $.
This proves that $ u(\mathbf x) \ge u_A(\mathbf x) $.
\end{proof}

\begin{remark}\label{sub-action and summable variantions}
If we keep the previous hypotheses when consindering a potential $ A \in C^0 (\boldsymbol{\Sigma}) $ with summable variations, we still
obtain a minimal, non-negative and bounded sub-action $ u_A: \boldsymbol{\Sigma} \to \mathbb R_+ $. Nevertheless, from
$$ \text{Var}_k (u_A) \le \sum_{j \ge k} \text{Var}_j (A), $$
we only assure its uniform continuity.
\end{remark}

It is important to notice that the existence of a sub-action as above indicates where we shall look for maximizing probabilities in the coercive case. 

\begin{proposition}\label{sub-action and finite alphabet}
Let $ (\boldsymbol{\Sigma}, \sigma) $ be a primitive Markov subshift on a countable alphabet. Suppose $ u \in C^0 (\boldsymbol{\Sigma}) $
is a bounded sub-action for a bounded above and coercive  potential $ A \in C^0 (\boldsymbol{\Sigma}) $. If $ \mu \in \mathcal M_\sigma $ is an
$A$-maximizing probability, then $ \mu $ is supported in a Markov subshift on a finite alphabet.
\end{proposition}

\begin{proof}
Let $ \mu \in \mathcal M_\sigma $ be an $A$-maximizing probability. Since $ u \in C^0 (\boldsymbol{\Sigma}) $ is a sub-action for the potential
$ A $, we have
$$ A + u - u \circ \sigma - \beta_A \le 0 \;\; \text{ and } \; \int (A + u - u \circ \sigma - \beta_A) \; d \mu = 0. $$
Therefore, the support of $ \mu $ is a subset of the closed set $ (A + u - u \circ \sigma - \beta_A)^{-1} (0) $.

Let $ \eta > 0 $ be a real constant. As $ A $ is coercive and $ u $ is bounded, there exists $ \hat I \in \mathbb Z_+ $ such that
\begin{equation}\label{coerciveness first time} 
\sup (A + u - u \circ \sigma - \beta_A)|_{\bigcup_{i > \hat I} [i]} < -\eta. 
\end{equation}
In particular, we obtain $ \mu (\bigcup_{i > \hat I} [i]) = 0 $, or in a more useful way $ \text{supp}(\mu) \subseteq \bigcup_{i \le \hat I} [i] $.

Being $ \text{supp}(\mu) $ a $\sigma$-invariant set, we get 
$ \text{supp}(\mu) \subseteq \bigcap_{k \ge 0} \sigma^{-k} \left(\bigcup_{i \le \hat I} [i] \right) = \Sigma_{\hat I} $,
which ends the proof.
\end{proof}

\begin{remark}\label{melhoria}
In \cite{JMU1}, when considering a primitive subshift on a countable alphabet, the authors showed there exist
invariant probabilities that maximize the integral of a bounded above and coercive potential $ A $ with summable variations and satisfying 
$ \inf A|_{\bigcup_{i \in \mathbb F} [i]} > - \infty $. They characterized them by the fact that their support lies in a compact subset of 
$ \boldsymbol{\Sigma} $. Remark~\ref{sub-action and summable variantions} and proposition~\ref{sub-action and finite alphabet} go beyond 
guaranteeing that those $A$-maximizing probabilities are actually supported in a Markov subshift on a finite alphabet.  
\end{remark}

\end{subsection}

\begin{subsection}*{Results for compact approximations}

In the context of a transitive expanding transformation defined on a compact metric space, the theory of ergodic optimization
has received special attention, which has yielded a more detailed theorical picture when the potential is sufficiently regular
as, let us say, Lipschitz continuous (see, for instance, \cite{CG, CLT, GL, Jenkinson}).
To demonstrate theorem~\ref{existence maximizing}, we will take advantage of results concerning ergodic optimal problems 
for the compact approximations $ (\Sigma_I, \sigma) $.

We suppose henceforth that $ (\boldsymbol{\Sigma}, \sigma) $ is a finitely primitive and
$ A: \boldsymbol{\Sigma} \to \mathbb R $ is a bounded above and locally H\"older continuous potential.
Recall (from footnote~\ref{obviedade}) that in this case $ \inf A|_{\bigcup_{i \in \mathbb F} [i]} > - \infty $. 

For $ I \ge I_{\mathbb F} $, we will need to consider the following ergodic constants
$$ \beta_A (I): = \max_{\substack{\mu \in \mathcal M_\sigma \\ \text{supp} \mu \subseteq \Sigma_I}} \int_{\Sigma_I} A \; d \mu. $$
Each one corresponds to the ergodic maximizing value associated to the Lipschitz continuous potential 
$ A|_{\Sigma_I} $ defined on the compact metric space $ \Sigma_I $. Recall $ \Sigma_I $ is the Markov 
subshift on the finite alphabet $ \{\iota_1, \ldots, \iota_{r_I}\} := \{0, \ldots, I\} \cap \left( \bigcap_{n \ge 0} \mathcal B_n \right) $ 
associated to the transition matrix $ \mathbf M|_{\{0, \ldots, I\} \times \{0, \ldots, I\}} $. If $ I \ge I_{\mathbb F} $, then
obviously $ \mathbb F \subset \{\iota_1, \ldots, \iota_{r_I}\} $ and $ (\Sigma_I, \sigma) $ is a topologically mixing dynamical system.

Remember that, in ergodic optimization on compact spaces, we call \emph{sub-action} for the potential $ A|_{\Sigma_I} $ any function $ u \in C^0(\Sigma_I) $ satisfying, for each point $ \mathbf x \in \Sigma_I $, $ A(\mathbf x) + u(\mathbf x) - u \circ \sigma (\mathbf x) \le \beta_A(I) $. Besides, a sub-action
$ u \in C^0(\Sigma_I) $ is said to be \emph{calibrated} when, for every $ \mathbf x \in \Sigma_I $, one can find a point $ \bar{\mathbf x} \in \Sigma_I $,
with $ \sigma(\bar{\mathbf x}) = \mathbf x $, such that
$$ A(\bar{\mathbf x}) + u(\bar{\mathbf x}) - u (\mathbf x) = \beta_A(I). $$
Main properties of calibrated sub-actions are discussed, for instance, in \cite{CLT, GL, Jenkinson}.

\begin{lemma}\label{oscillation calibrated}
Assume $ (\boldsymbol{\Sigma}, \sigma) $ is a finitely primitive Markov subshift on a countable alphabet.
Let $ A: \boldsymbol{\Sigma} \to \mathbb R $ be a bounded above and locally H\"older continuous potential. 
Consider an integer $ I \ge I_{\mathbb F} $.
If $ u \in C^0(\Sigma_I) $ is a calibrated sub-action for the Lipschitz continuous potential $ A |_{\Sigma_I} $, then 
$$ \text{osc}(u) := \max_{\mathbf x, \mathbf y \in \Sigma_I} [u(\mathbf x) - u(\mathbf y)] \le 
\text{Var}(A) + K_0 \left( \sup A - \inf_{i \in \mathbb F} A|_{[i]} \right). $$
\end{lemma}
  
\begin{proof}
Take arbitrary points $ \mathbf x, \mathbf y \in \Sigma_I $.
As $ u $ is a calibrated sub-action, we define inductively a sequence $ \{\mathbf x^k = (x^k_0, x^k_1, \ldots) \} \subseteq \Sigma_I $ 
by choosing $ \mathbf x^0 := \mathbf x $ and, for all $ k \ge 0 $, demanding $ \sigma(\mathbf x^{k + 1}) = \mathbf x^k $ with
$ u(\mathbf x^k) = u(\mathbf x^{k +1}) + A(\mathbf x^{k + 1}) - \beta_A(I) $.

Write  $ \mathbf y^0 := \mathbf y = (y_0, y_1, \ldots) $. 
Since$ (\boldsymbol{\Sigma}, \sigma) $ is finitely primitive and $ I \ge I_{\mathbb F} $, there exists a word 
$ (w_1, w_2, \ldots, w_{K_0}) \in \mathbb F^{K_0} $, with $ \mathbf M(w_j, w_{j + 1}) = 1 $, such that 
$ \mathbf M(x_0^{K_0 + 1}, w_1) = 1 = \mathbf M(w_{K_0}, y_0) $.
So we may consider the point $ \mathbf y^k \in \Sigma_I $ defined by 
$$ \mathbf y^k =
\left\{ \begin{array}{ll} (w_{K_0 - k + 1}, \ldots, w_{K_0}, y_0, y_1, \ldots) & \mbox{if $ 1 \le k \le K_0 $} \\ 
(x_0^k, \ldots, x_0^{K_0 + 1}, w_1, w_2, \ldots, w_{K_0}, y_0, y_1, \ldots) & \mbox{if $ k > K_0 $} \end{array} \right.. $$
Clearly, $ \sigma(\mathbf y^{k + 1}) = \mathbf y^k $ and
$ u(\mathbf y^k) \ge u(\mathbf y^{k +1}) + A(\mathbf y^{k + 1}) - \beta_A(I) $.

Then notice that
\begin{eqnarray*}
u(\mathbf x) - u(\mathbf y) & \le & u(\mathbf x^1) - u(\mathbf y^1) + A(\mathbf x^1) - A(\mathbf y^1) \\
& \le & u(\mathbf x^2) - u(\mathbf y^2) + A(\mathbf x^1) - A(\mathbf y^1) + A(\mathbf x^2) - A(\mathbf y^2) \\
& \vdots & \nonumber \\
& \le & u(\mathbf x^k) - u(\mathbf y^k) + \sum_{j = 1}^k [A(\mathbf x^j) - A(\mathbf y^j)]. 
\end{eqnarray*}

As $ d(\mathbf x^k, \mathbf y^k) = \lambda^{k - K_0 - 1} d(\mathbf x^{K_0 + 1}, \mathbf y^{K_0 + 1}) $ for $ k > K_0 $, the continuity of $ u $ 
implies $ \lim_{k \to \infty} [u(\mathbf y^k) - u(\mathbf x^k)] = 0 $. Hence, we obtain
\begin{eqnarray*}
u(\mathbf y) - u(\mathbf x) 
& \le & \sum_{j = 1}^\infty [A(\mathbf x^j) - A(\mathbf y^j)] \\
& = & \sum_{j = 1}^{K_0} [A(\mathbf x^j) - A(\mathbf y^j)] + \sum_{j = K_0 + 1}^\infty [A(\mathbf x^j) - A(\mathbf y^j)] \\
& \le & K_0 \left( \sup A - \inf A|_{\bigcup_{i \in \mathbb F} [i]} \right) + \text{Var}(A),
\end{eqnarray*}
from which the statement follows immediately.
\end{proof}

It is necessary to recall other central notions and facts of ergodic optimization on compact spaces.
A point $ \mathbf x \in \Sigma_I $ is said to be non-wandering with respect to the Lipschitz continuous
potential $ A|_{\Sigma_I} $ if, for all $ \epsilon > 0 $, one can find a point $ \mathbf y \in \Sigma_I $
and an integer $ n > 0 $ such that
$$ d(\mathbf x, \mathbf y) < \epsilon, \; d(\mathbf x, \sigma^n(\mathbf y)) < \epsilon \; \text{ and } \; | S_n (A - \beta_A(I))(\mathbf y) | < \epsilon. $$
Let $ \Omega(A, I) \subseteq \Sigma_I $ denote the set of non-wandering points with respect to $ A|_{\Sigma_I} $.

This set is a compact $\sigma$-invariant subset of $ \Sigma_I $. For any sub-action $ u \in C^0(\Sigma_I) $, 
\begin{equation}\label{Aubry and contact}
\Omega(A, I) \subseteq \{\mathbf x \in \Sigma_I : (A + u - u \circ \sigma - \beta_A(I))(\mathbf x) = 0 \}. 
\end{equation}
Furthermore, $ \Omega(A, I) $ characterizes the maximizing probabilities in the sense that, for $ \mu \in \mathcal M_\sigma $ 
with $ \text{supp}(\mu) \subseteq \Sigma_I $, one has
\begin{equation}\label{Mather and Aubry}
\int_{\Sigma_I} A \; d\mu = \beta_A(I) \; \Leftrightarrow \; \text{supp}(\mu) \subseteq \Omega(A, I). 
\end{equation}

The demonstrations of these properties and more details on the non-wandering set with respect to a Lipschitz continuous potential
may be found, for instance, in \cite{CLT, GL, GLT, LT}.

Since $ (\Sigma_{I_{\mathbb F}}, \sigma) $ is a topologically mixing dynamical system, we may consider a probability
measure $ \mu_{\mathbb F} \in \mathcal M_\sigma $ whose support is a periodic orbit in $ \Sigma_{I_{\mathbb F}} \cap {\mathbb F}^{\mathbb Z_+} $. 
In particular, for all $ I \ge I_{\mathbb F} $, notice that
\begin{equation}\label{lower bound muF}
\beta_A(I) \ge \int_{\Sigma_I} A \; d\mu_{\mathbb F} \ge \inf A|_{\bigcup_{i \in \mathbb F} [i]}. 
\end{equation}

Let us assume in addition that the potential $ A : \boldsymbol{\Sigma} \to \mathbb R $ is coercive. 
A fundamental inequality is thus the following one.

\begin{notation}
The coerciveness of the potential allows us to determine an integer $ \hat I > I_{\mathbb F} $ satisfying
\begin{equation}\label{definition Ihat}
\sup A|_{\bigcup_{i > \hat I} [i]} < \inf A|_{\bigcup_{i \in \mathbb F} [i]} - \left[ \text{Var}(A) + K_0 \left( \sup A - \inf A|_{\bigcup_{i \in \mathbb F} [i]} \right) \right].
\end{equation}
\end{notation}

So we have an important lemma.

\begin{lemma}\label{fixed beta}
Suppose $ (\boldsymbol{\Sigma}, \sigma) $ is a finitely primitive Markov subshift on a countable alphabet.
Let $ A: \boldsymbol{\Sigma} \to \mathbb R $ be a bounded above, coercive and locally H\"older continuous potential. Then
$$ \beta_A(I) = \beta_A(\hat I) \;\;\; \forall \; I \ge \hat I, $$
where the positive integer $ \hat I $ is defined by~(\ref{definition Ihat}).
Furthermore, given an integer $ I \ge \hat I $, only $(A|_{\Sigma_{\hat I}})$-maximizing probabilities 
maximize the integral of $ A|_{\Sigma_I} $ among $\sigma$-invariant probabilities supported in $ \Sigma_I $.
\end{lemma}

\begin{proof}
Clearly $ \beta_A(\hat I) \le \beta_A(I) $ whenever $ I \ge \hat I $. In order to obtain the equality, 
it is enough to show that every $(A|_{\Sigma_I})$-maximizing probability is actually supported in $ \Sigma_{\hat I} $.

Suppose on the contrary the existence of a probability measure $ \mu \in \mathcal M_\sigma $, with $ \text{supp}(\mu) \subseteq \Sigma_I $ and
$ \int_{\Sigma_I} A \; d\mu = \beta_A(I) $, such that $ \text{supp}(\mu) - \Sigma_{\hat I} \neq \emptyset $.

Take then $ \mathbf x = (x_0, x_1, \ldots) \in \text{supp}(\mu) - \Sigma_{\hat I} $. We may assume $ x_0 > \hat I $. 
Therefore, from~(\ref{lower bound muF}) and~(\ref{definition Ihat}), it follows
\begin{eqnarray*}
A(\mathbf x) - \beta_A(I) & \le & \sup A|_{\bigcup_{i > \hat I} [i]} - \inf A|_{\bigcup_{i \in \mathbb F} [i]} \\
& < & - \left[ \text{Var}(A) + K_0 \left( \sup A - \inf A|_{\bigcup_{i \in \mathbb F} [i]} \right) \right].
\end{eqnarray*}

Let $ u \in C^0(\Sigma_I) $ be a calibrated sub-action for the Lipschitz continuous potential $ A|_{\Sigma_I} $.
Thanks to~(\ref{Mather and Aubry}) and~(\ref{Aubry and contact}), we have
$ A(\mathbf x) + u(\mathbf x) - u \circ \sigma(\mathbf x) - \beta_A(I) = 0 $,
which then yields
$$ u(\mathbf x) - u \circ \sigma(\mathbf x) > \text{Var}(A) + K_0 \left( \sup A - \inf A|_{\bigcup_{i \in \mathbb F} [i]} \right). $$
 
However, this inequality contradicts lemma~\ref{oscillation calibrated} which assures that
$$ \text{osc}(u) \le \text{Var}(A) + K_0 \left( \sup A - \inf A|_{\bigcup_{i \in \mathbb F} [i]} \right). $$

Hence, necessarily $ \text{supp}(\mu) \subseteq \Sigma_{\hat I} $ whenever $ \mu \in \mathcal M_\sigma $ maximizes
the integral of $ A|_{\Sigma_I} $ among the $\sigma$-invariant probabilities supported in $ \Sigma_I $.
\end{proof}

\end{subsection}

\begin{subsection}*{Proof of theorem~\ref{existence maximizing}}

Our strategy is to extend the statement of lemma~\ref{fixed beta} to the noncompact dynamical system $ (\boldsymbol{\Sigma}, \sigma) $.
More precisely, we will show that
\begin{equation}\label{central inequality}
\int A \; d\mu \le \beta_A(\hat I), \;\; \forall \; \mu \in \mathcal M_\sigma. 
\end{equation}

Clearly it will follow $ \beta_A = \beta_A(\hat I) $, guaranteeing the existence of maximizing probabilities.
Propositions~\ref{minimal sub-action} and~\ref{sub-action and finite alphabet} and lemma~\ref{fixed beta} will then assure that
only $A|_{\Sigma_{\hat I}}$-maximizing probabilities maximize the integral of the potential $ A $ among all $\sigma$-invariant Borel probability measures.
Besides, from~(\ref{Mather and Aubry}), the compact $\sigma$-invariant subset of $ \Sigma_{\hat I} $ in the statement of 
theorem~\ref{existence maximizing} will immediately be $ \Omega = \Omega(A, \hat I) $.

So we just need to demonstrate~(\ref{central inequality}). As a matter of fact, this inequality is a consequence of the denseness of probabilities
whose support is a pediodic orbit (see, for instance, \cite{Parthasarathy}) and lemma~\ref{fixed beta}. For the sake of completeness, we discuss its
proof carefully.

Notice first that, thanks to the ergodic decomposition theorem, it is enough to suppose $ \mu \in \mathcal M_\sigma $ ergodic.
It seems convenient to recall that as usual we are considering the space of bounded real-valued functions on $ \boldsymbol{\Sigma} $
and its subspaces equipped with the uniform norm. We take then a dense sequence $ \{f_j\}_{j \ge 0} $ of bounded uniformly continuous 
real-valued functions on $ \boldsymbol{\Sigma} $. Let $ \Lambda_j \subseteq \boldsymbol{\Sigma} $ denote the set of points for which the 
Birkhoff's ergodic theorem holds for $ f_j $ as a $\mu$-integrable function. Take then a point $ \mathbf z \in \bigcap_{j \ge 0} \Lambda_j $. 
It is not difficult to see that the sequence of Borel probability measures 
$$ \nu_k := \frac{1}{k} \sum_{j = 0}^{k - 1} \delta_{\sigma^j(\mathbf z)} $$
converges in the weak topology to $ \mu $.   

Since $ (\boldsymbol{\Sigma}, \sigma) $ is a finitely primitive Markov subshift, for every integer $ k > K_0 $, 
let $ \mathbf y^k = (y_0^k, y_1^k, \ldots) \in \boldsymbol{\Sigma} $ be a periodic point of period $ k $, 
with $ y_j^k \in \mathbb F $ whenever $ j \in \{k - K_0, \ldots, k - 1\} $, such that 
$ d(\mathbf z, \mathbf y^k) \le \lambda^{k - K_0} $. Consider then the $\sigma$-invariant Borel probability measure
$$ \mu_k := \frac{1}{k} \sum_{j = 0}^{k - 1} \delta_{\sigma^j(\mathbf y^k)} \in \mathcal M_\sigma. $$
Let $ f: \boldsymbol{\Sigma} \to \mathbb R $ be a bounded function dependending on $ n $ coordinates, that is, satisfying $ \text{Var}_n(f) = 0 $.
Notice that (supposing $ k > K_0 + n $)
$$ \left | \int f \; d\mu_k - \int f \; d\nu_k \right | = \frac{1}{k}\left|S_k f(\mathbf y^k) - S_k f(\mathbf z) \right| \le 
\frac{2}{k} (K_0 + n) \| f \|_{\infty} \rightarrow 0 \; \text{ as } \; k \rightarrow \infty. $$

As functions depending on finitely many coordinates are dense among bounded uniformly continuous real-valued 
functions on $ \boldsymbol{\Sigma} $, we conclude that the sequences $ \{\mu_k\} $ and $ \{\nu_k\} $ have the same weak limit $ \mu $. However,
lemma~\ref{fixed beta} assures that, for each index $ k $,
$$ \int A \; d\mu_k \le \beta_A(\hat I). $$
Thus, (\ref{central inequality}) follows just by passing to the limit.
\end{subsection}

\begin{subsection}*{A final remark}

Notice that, in reality, the coerciveness of the potential was exactly used twice in our arguments. Indeed, the coercive condition
was employed just to assure both inequalities~(\ref{coerciveness first time}) and~(\ref{definition Ihat}).

Nevertheless, during the construction of the sub-action $ u_A \in C^0(\boldsymbol{\Sigma}) $ in the proof of proposition~\ref{minimal sub-action},
its boundness was made explicit in~(\ref{boundness uA}). Therefore, one can easily adapted the demonstration of proposition~\ref{sub-action and finite alphabet}
using this information and the fact that $ \beta_A \ge \inf A |_{\bigcup_{i \in \mathbb F} [i]} $ in order to guarantee the following statement.

\begin{proposition}
Let $ (\boldsymbol{\Sigma}, \sigma) $ be a finitely primitive Markov subshift on a countable alphabet.
Assume $ A: \boldsymbol{\Sigma} \to \mathbb R $ is a bounded above and locally H\"older continuous potential. 
Suppose there exists an integer $ \hat I > I_{\mathbb F} $ such that
$$ \sup A|_{\bigcup_{i > \hat I} [i]} < \inf A|_{\bigcup_{i \in \mathbb F} [i]} - \left[ \text{Var}(A) + K_0 \left( \sup A - \inf A|_{\bigcup_{i \in \mathbb F} [i]} \right) \right]. $$ 
Then, $ \text{supp}(\mu) \subseteq \Sigma_{\hat I} $ whenever $ \mu \in \mathcal M_\sigma $ is an $A$-maximizing probability.
\end{proposition}

Since lemma~\ref{fixed beta} is actually a consequence of inequality~(\ref{definition Ihat}) and not of the coerciveness of the potential,
one may now obtain a more general version of theorem~\ref{existence maximizing}, without necessarily imposing an asymptotic behavior to $ \sup A |_{[i]} $.
In fact, we have the following result.

\begin{theorem}
Suppose $ (\boldsymbol{\Sigma}, \sigma) $ is a finitely primitive Markov subshift on a countable alphabet.
Let $ A: \boldsymbol{\Sigma} \to \mathbb R $ be a bounded above and locally H\"older continuous potential. 
Assume the existence of an integer $ \hat I > I_{\mathbb F} $ such that
$$ \sup A|_{\bigcup_{i > \hat I} [i]} < \inf A|_{\bigcup_{i \in \mathbb F} [i]} - \left[ \text{Var}(A) + K_0 \left( \sup A - \inf A|_{\bigcup_{i \in \mathbb F} [i]} \right) \right]. $$ 
Then $ \beta_A = \beta_A(\hat I) $. Moreover, $ \mu \in \mathcal M_\sigma $ is an $A$-maximizing probability if, and only if, 
$ \text{supp}(\mu) \subseteq \Omega(A, \hat I) $. 
\end{theorem}

We decided to discuss this generalized result at the end of the paper because the existence of $ \hat I $ in the above statement
seems to be just a technical assumption. Coerciveness, in turn, is compelling, as the works in thermodynamic formalism indicate.
Besides, it is important to have in mind that certain maximizing probabilities can be seen as zero temperature limits of Gibbs-equilibrium states
(see \cite{JMU0, Morris}).

Finally, we would like to point out that inequality~(\ref{definition Ihat}), which has proved to be so fundamental, is quite similar to the oscillation 
condition proposed in \cite{JMU1} (see definition 5.1 there). It is interesting to refind such a condition as a natural consequence of uniform oscillatory 
behaviour of calibrated sub-actions defined on compact approximations. 

\end{subsection}

\end{section}

\end{document}